\documentclass[12pt]{article}
\usepackage{fancyhdr,graphicx}
\usepackage{amsfonts}
\usepackage{amssymb}
\usepackage{amsthm}
\usepackage{newlfont}
\usepackage{amsmath}
\usepackage{ulem}
\usepackage{cancel}
\usepackage{color}
\usepackage[numbers,sort&compress]{natbib}
\newtheorem{thm}{Theorem}[section]
\newtheorem{Lemma}[thm]{Lemma}
\newtheorem{cor}[thm]{Corollary}
\newtheorem{pro}[thm]{Proposition}

\usepackage{latexsym,bm}
\title{{Resistance distance-based graph invariants of subdivisions and triangulations of graphs}}
\author{Yujun Yang$^{1,2}$, Douglas J. Klein$^{1}$
\\\small{1. Department of Marine Sciences, Texas A\&M University at Galveston,}
\\\small{Galveston, Texas, 77553-1675}
\\\small{2. School of Mathematics and Information Science, Yantai University,}
\\\small{Yantai, Shandong, 264005, P.R. China}
\\\small{E-mail addresses: yangy@tamug.edu, kleind@tamug.edu}}

\date{}

\begin{document}

\maketitle \baselineskip 18pt
\begin{abstract}
We study three resistance distance-based graph invariants: the Kirchhoff index, and two modifications, namely, the multiplicative degree-Kirchhoff index and the additive degree-Kirchhoff index. In work in press, one of the present authors (2014) and Sun et al. (2014) independently obtained (different) formulas for the Kirchhoff index of subdivisions of graphs. Huang et al. (2014) obtained a formula for the Kirchhoff index of triangulations of graphs. In our paper, first we derive formulae for the additive degree-Kirchhoff index and the multiplicative degree-Kirchhoff index of subdivisions and triangulations, as well as a new formula for the Kirchhoff index of triangulations, in terms of invariants of $G$. Then comparisons are made between each of our Kirchhoffian graph invariants for subdivision and triangulation. Finally, formulae for these graph invariants of iterated subdivisions and triangulations of graphs are obtained.
\vspace{0.5cm}

\noindent \textbf{Key words:} resistance distance, Kirchhoff index, subdivision, triangulation, additive degree-Kirchhoff index, multiplicative degree-Kirchhoff index
\end{abstract}

%%introduction--------------------------------------------------
\section{Introduction}
Distance based graph invariants, such as the Wiener index, and the Szeged index, have been widely studied (see, e.g. \cite{ah,cll,kn,kn1,kya} and references therein). In 1993, a new distance function, named resistance distance \cite{kr}, was identified as an alternative of the ordinary (shortest path) distance. This new intrinsic graph metric, which comes from electrical network theory and generalizes the ordinary distance to some extent, turns out to have many nicely pure mathematical interpretations \cite{cz,ds,k,ki,nash,sh,ss,ss1,xg}. Since then, resistance distance, and invariants based on it, have been extensively studied.

Let $G=(V(G),E(G))$ be a connected graph. The \textit{resistance distance} \cite{s1,s2,s3,gg,kr} between a pair of vertices $i$ and $j$, denoted by $\Omega_{ij}$, is the net effective resistance measured across nodes $i$ and $j$ in the electrical network constructed from $G$ by replacing each edge
with a unit resistor.

Analogous to distance-based graph invariants, various graph invariants based on resistance distance have been defined and studied. Among these invariants, the most famous one is the \textit{Kirchhoff index} \cite{kr}, also known as the \textit{total effective resistance} \cite{gbs} or the \textit{effective graph resistance} \cite{esv}, which is denoted by $R(G)$ and defined as the sum of resistance distances between all pairs of vertices of $G$, i.e.
\begin{equation}
R(G)=\sum_{\{i,j\}\subseteq V}\Omega_{ij}.
\end{equation}
Much attention has been given in recent years to this index. For more information, the readers are referred to most recent papers \cite{bcpt,ch2, das,dc,dc1,ns,ssn,yyh,yk,zzh} and references therein.

Recently, two modifications of the Kirchhoff index, which takes the degrees of the graph into account, have been considered.
One is the \textit{multiplicative degree-Kirchhoff index} defined by Chen and Zhang \cite{cz1}:
\begin{equation}
R^*(G)=\sum_{\{i,j\}\subseteq V}d_id_j\Omega_{ij},
\end{equation}
where $d_i$ is the degree (i.e., the number of neighbors) of the vertex $i$. The other one is the \textit{additive degree-Kirchhoff index} defined by Gutman et al. \cite{gf}:
\begin{equation}
R^+(G)=\sum_{\{i,j\}\subseteq V}(d_i+d_j)\Omega_{ij}.
\end{equation}
For more work on these two modifications, the readers are referred to recent papers \cite{bcpt1,fgy,nag,pa,pr,yang}.

The \textit{subdivision} of $G$, denoted by $S(G$), is the graph obtained by replacing every edge in $G$ with a copy of $P_2$ (path of length two). The \textit{triangulation} \cite{yy,rk} of $G$, denoted by $T(G)$, is the graph obtained from $G$ by changing each edge $uv$ of $G$ into a triangle $uwv$ with $w$ the new vertex associated with $uv$.
For example, the subdivision and the triangulation of the five-vertex complete graph $K_5$ are shown in Figure 1.

\begin{figure}[h]
\begin{center}
\includegraphics[scale=0.70
]{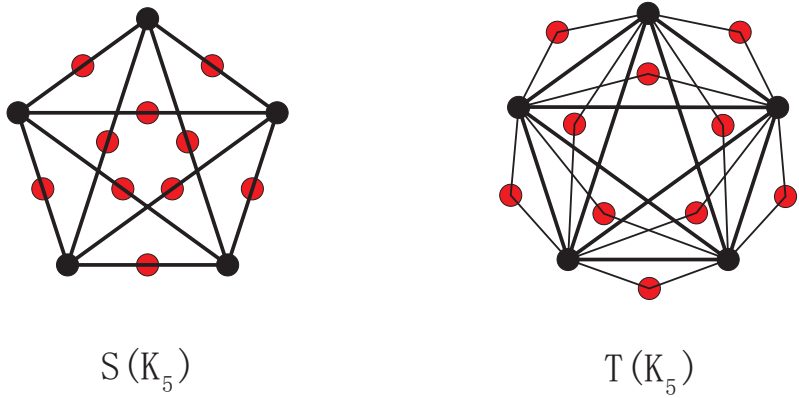}\caption{Graphs $S(K_5)$ and $T(K_5)$.}
\end{center}
\end{figure}

In \cite{gl}, Gao et al. obtained a formula for the Kirchhoff index of $S(G)$ for a regular graph $G$. Then one of the present authors \cite{yang}, and Sun et al. \cite{swzb} independently extended it to general graphs, with $R(S(G))$ being expressed in different ways. In \cite{yang}, it is showed that for a general graph $G$, the Kirchhoff index of $S(G)$ could be expressed in terms of $R(G)$, $R^+(G)$, $R^*(G)$, $|V(G)|$, and $|E(G)|$. For the triangulation of a regular graph $G$, Wang et al. \cite{wyl} obtained a formula for $R(T(G))$. Then Huang et al. \cite{hzb} generalized their results to general graphs, though this formula features the group inverse $L^\#(G)$ of the Laplacian matrix of $G$ in their expression. In this paper, we obtain a new formula for $R(T(G))$, expressed in terms of ordinary graph invariants of $G$, much as for $S(G)$. In addition, formulae for the additive degree-Kirchhoff index and the multiplicative degree-Kirchhoff index of $S(G)$ and $T(G)$ are obtained in terms of the same graph invariants of $G$. From these results, for each graph invariant $\mathbb{I}$ ($\mathbb{I}\in \{R,R^+,R^*\}$), a comparison between $\mathbb{I}(S(G))$ and $\mathbb{I}(R(G))$ is obtained, to show that the $\mathbb{I}(S(G))$ is a linear function of the $\mathbb{I}(R(G))$. Finally, formulae for the three resistance distance-based graph invariants of the $k$-th iterated subdivision and triangulation of $G$ are also obtained.

\section{Formulae for $R^+(S(G))$ and $R^*(S(G))$}
Let $G=(V(G),E(G))$ be a connected graph with $n$ vertices and $m$ edges ($n\geq 2$). In what follows, for simplicity, we use $V$ and $E$ to denoted $V(G)$ and $E(G)$, respectively. We suppose that $|V|=n$ and $|E|=m$. For $i\in V$, we use $\Gamma(i)$ to denote the neighbor set of $i$ in $G$. Since the subdivision graph $S(G)$ is the graph obtained by inserting an additional vertex in each edge of $G$, the vertex set $V(S(G))$ of $S(G)$ may be written as $V(S(G))=V\cup V'$, where $V'$ denotes the set of inserted vertices. Clearly $|V'|=|E|=m$ and $|V(S(G))|=n+m$. In the following, for convenience, we use $\Omega^S_{ij}$ to denote the resistance distance between $i$ and $j$ in $S(G)$.

In \cite{cz}, Chen and Zhang gave a complete characterization to resistance distances in $S(G)$ in terms of resistance distances in $G$. Their result, as given in the following lemma, plays an essential rule.
\begin{Lemma}\cite{cz}\label{sg}
Resistance distances in $S(G)$ can be computed as follows:

(1) For $i,j\in V$,  $$\Omega_{ij}^S=2\Omega_{ij}.$$

(2) For $i\in V'$, $\Gamma(i)=\{k,l\}$, $$\Omega_{ik}^S=\Omega_{il}^S=\frac{1+\Omega_{kl}}{2}.$$

(3) For $i\in V'$, $\Gamma(i)=\{k,l\}$, and $j\in V$, $j\neq k, l$,
$$\Omega_{ij}^S=\frac{1+2\Omega_{kj}+2\Omega_{lj}-\Omega_{kl}}{2}.$$

(4) For $i,j\in V'$, $\Gamma(i)=\{k,l\}$, $\Gamma(j)=\{p,q\}$,
$$\Omega_{ij}^S=\frac{2+\Omega_{pk}+\Omega_{qk}+\Omega_{pl}+\Omega_{ql}-\Omega_{kl}-\Omega_{pq}}{2}.$$
\end{Lemma}

In fact, it is easily observed that if $j=k$ or $j=l$, then (2) is deduced from (3). Hence for $i\in V'$, $\Gamma(i)=\{k,l\}$, and $j\in V$, we have
\begin{equation}
\Omega_{ij}^S=\frac{1+2\Omega_{kj}+2\Omega_{lj}-\Omega_{kl}}{2}.
\end{equation}

To obtain the main result, we also need to introduce Foster's (first) formula in electrical network theory.
\begin{Lemma}(Foster's formula)\cite{f}
Let $G$ be a connected graph with $n$ vertices. Then the sum of resistance distances between all pairs of adjacent vertices is equal to $n-1$, i.e.
\begin{equation}
\sum_{ij\in E}\Omega_{ij}=n-1,
\end{equation}
where the summation is taken over all the edges of $G$.
\end{Lemma}

In \cite{gl}, Gao et al. obtained formula for the Kirchhoff index of $S(G)$ for a regular graph $G$. Then one of the present authors \cite{yang}, and Sun et al. \cite{swzb} independently extended it to general graphs, with $R(S(G))$ being expressed in different ways. In \cite{yang}, it was shown that
\begin{thm}\label{main}\cite{yang}
Let $G$ be a connected graph with $n\geq 2$ vertices and $m$ edges. Then
\begin{equation}\label{e4}
R(S(G))=2R(G)+R^+(G)+\frac{1}{2}R^*(G)+\frac{m^2-n^2+n}{2}.
\end{equation}
\end{thm}

In the following, we compute the additive degree-Kirchhoff index and the multiplicative degree-Kirchhoff index of $S(G)$. For $i\in V(S(G))$, we use $d^S_i$ to denote the degree of $i$ in $S(G)$. Then it is obvious that for $i\in V$, $d_i^S=d_i$, and for $i\in V'$, $d^S_i=2$.

We can now give our formula for $R^+(S(G))$.
\begin{thm}\label{s+}
Let $G$ be a connected graph with $n\geq 2$ vertices and $m$ edges. Then
\begin{equation}\label{e4}
R^+(S(G))=4R^+(G)+4R^*(G)+(m+n)(m-n+1)+2m(m-n).
\end{equation}
\end{thm}
\begin{proof}
By the definition of the additive degree-Kirchhoff index, we have
\begin{align}\label{e5}
R^+(S(G))=&\sum_{\{i,j\}\subseteq (V\cup V')}(d_i^S+d_j^S)\Omega^S_{ij}\notag\\
=&\sum_{\{i,j\}\subseteq V}(d_i^S+d_j^S)\Omega^S_{ij}+\sum_{i\in V'}\sum_{j\in V}(d_i^S+d_j^S)\Omega^S_{ij}+\sum_{\{i,j\}\subseteq V'}(d_i^S+d_j^S)\Omega^S_{ij}
%=&\sum_{\{i,j\}\subseteq V}(d_i+d_j)\Omega^S_{ij}+\sum_{i\in V'}\sum_{j\in \Gamma(i)}(d_i+d_j)\Omega^S_{ij}+\sum_{i\in V'}\sum_{j\in %V\setminus\Gamma(i)}(d_i+d_j)\Omega^S_{ij}+\sum_{\{i,j\}\subseteq V'}(d_i+d_j)\Omega^S_{ij}.
\end{align}
Now we compute the three parts on the right side of the above equation.

\textbf{(1).} For the first part, by Lemma \ref{sg}, it is easily seen that
\begin{equation}\label{e12}
\sum_{\{i,j\}\subseteq V}(d_i^S+d_j^S)\Omega^S_{ij}=2\sum_{\{i,j\}\subseteq V}(d_i+d_j)\Omega_{ij}=2R^+(G).
\end{equation}

\textbf{(2).} For the second part, noticing that $d_i^S=2$ for $i\in V'$, one has
\begin{equation}\label{e13}
\sum_{i\in V'}\sum_{j\in V}(d_i^S+d_j^S)\Omega^S_{ij}=\sum_{i\in V'}\sum_{j\in V}(2+d_j)\Omega^S_{ij}=2\sum_{i\in V'}\sum_{j\in V}\Omega^S_{ij}+\sum_{i\in V'}\sum_{j\in V}d_j\Omega^S_{ij}.
\end{equation}
For the $\sum\limits_{i\in V'}\sum\limits_{j\in V}\Omega^S_{ij}$ term,
$$\sum_{i\in V'}\sum_{j\in V}\Omega^S_{ij}=\sum_{i\in V'}\sum_{j\in \Gamma(i)}\Omega^S_{ij}+\sum_{i\in V'}\sum_{j\in V\setminus\Gamma(i)}\Omega^S_{ij},$$
while \cite{yang} has proved that
$\sum\limits_{i\in V'}\sum\limits_{j\in \Gamma(i)}\Omega^S_{ij}=m+n-1$ and $\sum\limits_{i\in V'}\sum\limits_{j\in V\setminus\Gamma(i)}\Omega^S_{ij}=\frac{m(n-2)}{2}+R^+(G)-\frac{(n+2)(n-1)}{2}$, so that
\begin{equation}\label{e14}
\sum\limits_{i\in V'}\sum\limits_{j\in V}\Omega^S_{ij}=R^+(G)+\frac{mn-n^2+n}{2}.
\end{equation}

For the last $\sum\limits_{i\in V'}\sum\limits_{j\in V}d_j\Omega^S_{ij}$ term of (\ref{e13}), we combine Lemma \ref{sg} and Foster's formula, to give
\begin{align}\label{e15}
&\sum_{i\in V'}\sum_{j\in V}d_j\Omega^S_{ij}=d_j\times\frac{1+2\Omega_{kj}+2\Omega_{lj}-\Omega_{kl}}{2}\notag\\
=&\frac{1}{2}\sum_{i\in V'}\sum_{j\in V}d_j+\sum_{i\in V'}\sum_{j\in V}d_j(\Omega_{kj}+\Omega_{lj})-\frac{1}{2}\sum_{i\in V'}\sum_{j\in V}d_j\Omega_{kl}\notag\\
=&\frac{1}{2}\sum_{i\in V'}(2m)+\sum_{i\in V'}\sum_{j\in V}d_j(\Omega_{kj}+\Omega_{lj})-\frac{1}{2}\sum_{i\in V'}(2m\Omega_{kl})\notag\\
=&m^2+\sum_{i\in V'}\sum_{j\in V}d_j(\Omega_{kj}+\Omega_{lj})-m(n-1).
\end{align}
The only thing left is to compute $\sum\limits_{i\in V'}\sum\limits_{j\in V}d_j(\Omega_{kj}+\Omega_{lj})$ in Eq. (\ref{e15}). Since every vertex $i\in V'$ corresponds to an edge $kl\in E(G)$ (with $\Gamma(i)=\{k,l\}$), it follows that
\begin{align}\label{e16}
\sum\limits_{i\in V'}\sum\limits_{j\in V}d_j(\Omega_{kj}+\Omega_{lj})&=\sum\limits_{kl\in E}\sum\limits_{j\in V}d_j(\Omega_{kj}+\Omega_{lj})=\sum\limits_{j\in V}d_j\sum\limits_{kl\in E}(\Omega_{kj}+\Omega_{lj})\notag\\
&=\sum\limits_{j\in V}d_j\sum\limits_{u\in V}d_u\Omega_{uj}=\sum\limits_{j\in V}\sum\limits_{u\in V}d_jd_u\Omega_{ju}\notag\\
&=2R^*(G).
\end{align}
%Since the left summation is over all edges of $G$, it is not hard to see that for any $i\in V$, the term $\sum_{j\in V}d_j\Omega_{ij}$ appears exactly $d_i$ %times in the summation. Hence
%\begin{equation}
%\sum\limits_{kl\in E}\sum\limits_{j\in V}d_j(\Omega_{kj}+\Omega_{lj})=\sum_{i\in V}d_i\sum_{j\in V}d_j\Omega_{ij}=\sum_{\{i,j\}\subseteq %V}d_id_j\Omega_{ij}=2R^*(G).
%\end{equation}
Substituting Eq. (\ref{e16}) back into Eq. (\ref{e15}), simple calculation leads to
\begin{equation}\label{e17}
\sum_{i\in V'}\sum_{j\in V}d_j\Omega^S_{ij}=2R^*(G)+m^2-m(n-1).
\end{equation}

Finally, substituting Eqs. (\ref{e17}) and (\ref{e14}) back into Eq. (\ref{e13}), we have
\begin{equation}\label{f1}
\sum_{i\in V'}\sum_{j\in V}(d_i+d_j)\Omega^S_{ij}=2R^+(G)+2R^*(G)+(m+n)(m-n+1).
\end{equation}

\textbf{(3).} For the third part of (\ref{e5}), we use a result of \cite{yang} that
\begin{align}\label{e18}
\sum_{\{i,j\}\subseteq V'}\Omega^S_{ij}&=\frac{m(m-1)}{2}+\frac{R^*(G)-(n-1)}{2}-\big[\frac{(m-2)(n-1)}{4}+\frac{m(n-1)}{4}\big]\notag\\
&=\frac{R^*(G)}{2}+\frac{m(m-n)}{2}.
\end{align}
Noticing that for $i\in V'$, $d_i^S=2$, it is seen that
\begin{equation}\label{e19}
\sum_{\{i,j\}\subseteq V'}(d_i^S+d_j^S)\Omega^S_{ij}=4\sum_{\{i,j\}\subseteq V'}\Omega^S_{ij}=2R^*(G)+2m(m-n).
\end{equation}

Substitution of the three parts (\ref{e12}), (\ref{f1}), and (\ref{e19}) into (\ref{e5}) yields the desired result.
\end{proof}

We proceed to obtain $R^*(S(G))$.
\begin{thm}\label{s*}
Let $G$ be a connected graph with $n\geq 2$ vertices and $m$ edges. Then
\begin{equation}\label{e}
R^*(S(G))=8R^*(G)+2m(2m-2n+1).
\end{equation}
\end{thm}
\begin{proof}
By the definition of the multiplicative degree-Kirchhoff index, we have
\begin{align*}
R^*(S(G))&=\sum_{\{i,j\}\subseteq V\cup V'}d_i^Sd_j^S\Omega^S_{ij}=\sum_{\{i,j\}\subseteq V}d_i^Sd_j^S\Omega^S_{ij}+\sum_{i\in V'}\sum_{j\in V}d_i^Sd_j^S\Omega^S_{ij}+\sum_{\{i,j\}\subseteq V'}d_i^Sd_j^S\Omega^S_{ij}\notag\\
&=\sum_{\{i,j\}\subseteq V}d_id_j\Omega^S_{ij}+2\sum_{i\in V'}\sum_{j\in V}d_j\Omega^S_{ij}+4\sum_{\{i,j\}\subseteq V'}\Omega^S_{ij}.
\end{align*}
By Lemma \ref{sg}, Eqs. (\ref{e17}) and (\ref{e19}), this gives
\begin{align*}
R^*(S(G))&=2\sum_{\{i,j\}\subseteq V}d_id_j\Omega_{ij}+2[2R^*(G)+m^2-m(n-1)]+[2R^*(G)+2m(m-n)]\\
&=2R^*(G)+2[2R^*(G)+m^2-m(n-1)]+[2R^*(G)+2m(m-n)]\\
&=8R^*(G)+2m(2m-2n+1),
\end{align*}
as required.
\end{proof}
\section{Formulae for our Kirchhoffian invariants of iterated subdivisions}
Set $S^0(G)=G$, $S^k(G)=S(S^{k-1}(G))$, $k=1,2,\ldots.$ In this section, in the proofs, denote by $n_k$ and $m_k$ the numbers of vertices and edges in $S^k(G)$, respectively. Then it is easily seen that $n_0=n$, $m_0=m$, and for $k\geq 1$,
\begin{equation}
n_k=n+(2^k-1)m,~~m_k=2^km.
\end{equation}
In what follows, we give formulae for these graph invariants of the iterated subdivision $S^k(G)$.

The following lemma for solving a certain kind of recursion relation, is useful.
\begin{Lemma}\cite{yyz}\label{lema}
Let $\{y_k\}_{k\geq 0}$, $\{f_k\}_{k\geq 0}$, and $\{g_k\}_{k\geq 0}$ be three sequences satisfying the recurrence relation:
$$y_{k+1}=f_ky_k+g_k, ~~~k\geq 0.$$
Then
$$y_{k+1}=(y_0+\sum_{i=0}^{k}h_i)\prod_{j=0}^{k}f_j,$$
where $h_k=s_{k+1}g_k$, $s_{k+1}=(\prod\limits_{i=0}^{k}f_i)^{-1}$, and $s_0=1$.
\end{Lemma}

Now we are ready to give formulae for these graph invariants of $S^k(G)$. We begin with the multiplicative degree-Kirchhoff index, as it will be used in obtaining the formulas for the additive degree-Kirchhoff index and the Kirchhoff index.
\begin{thm}\label{sk*}
Let $G$ be a connected graph with $n\geq 2$ vertices and $m$ edges. Then
\begin{equation}\label{e}
R^*(S^k(G))=8^kR^*(G)+\frac{8^k-2^k}{3}m(2m-2n+1).
\end{equation}
\end{thm}
\begin{proof}
Let $n_k$ and $m_k$ be defined as before with $S^k(G)$ a graph of $n_k$ vertices and $m_k$ edges. By Theorem \ref{s*},
$$R^*(S^k(G))=8R^*(S^{k-1}(G))+2m_{k-1}(2m_{k-1}-2n_{k-1}+1),~~k\geq 1.$$
Set $y_k=R^*(S^k(G))$, $f_k=8$, $g_k=2m_k(2m_k-2n_k+1)$, $k=0,1,2,\ldots$. Then we have
$$y_{k+1}=f_ky_k+g_k,~~y_0=R^*(G),~~k\geq 0.$$
By Lemma \ref{lema},
$$y_{k+1}=(y_0+\sum_{i=0}^{k}h_i)\prod_{j=0}^{k}f_j,$$
where $s_0=1$, and for $k\geq 0$, $s_{k+1}=(\prod\limits_{j=0}^{k}f_j)^{-1}=(\prod\limits_{j=0}^{k}8)^{-1}=\frac{1}{8^{k+1}}$, and
$$h_k=s_{k+1}g_k=\frac{2m_k(2m_k-2n_k+1)}{8^{k+1}}=\frac{2\times 2^km(2m-2n+1)}{8^{k+1}}=\frac{2^{k+1}m(2m-2n+1)}{8^{k+1}}.$$
Hence
\begin{align*}
y_{k+1}&=[y_0+\sum_{i=0}^{k}\frac{2^{i+1}m(2m-2n+1)}{8^{i+1}}]\prod_{j=0}^{k}8\\
&=[R^*(G)+m(2m-2n+1)\sum_{i=0}^{k}(\frac{1}{4})^{i+1}]8^{k+1}\\
&=[R^*(G)+m(2m-2n+1)\frac{\frac{1}{4}-(\frac{1}{4})^{k+2}}{1-\frac{1}{4}}]8^{k+1}\\
&=8^{k+1}R^*(G)+\frac{8^{k+1}-2^{k+1}}{3}m(2m-2n+1).
\end{align*}
As $R(S^k(G))=y_k$, the proof is completed.
\end{proof}

Next, making use of the result of $R^*(S^k(G))$, we compute $R^+(S^k(G))$.

\begin{thm}\label{sk+}
Let $G$ be a connected graph with $n\geq 2$ vertices and $m$ edges. Then
\begin{equation}
R^+(S^k(G))=4^kR^+(G)+(8^k-4^k)R^*(G)+\frac{8^k-2^k}{3}m(2m-2n+1)-\frac{4^k-1}{3}(m-n)(m-n+1).
\end{equation}
\end{thm}
\begin{proof}
By Theorem \ref{s+},
\begin{align*}
R^+({S^k(G)})=&4R^+(S^{k-1}(G))+4R^*(S^{k-1}(G))+(m_{k-1}+n_{k-1})(m_{k-1}-n_{k-1}+1)\\
&+2m_{k-1}(m_{k-1}-n_{k-1}),\quad k\geq 1.
\end{align*}
Set $y_k=R^+(S^k(G))$, $f_k=4$, $g_k=4R^*(S^{k}(G))+(m_{k}+n_{k})(m_{k}-n_{k}+1)$, $k=0,1,2,\ldots$. Then we have
$$y_{k+1}=f_ky_k+g_k,~~y_0=R^+(G),~~k\geq 0.$$
By Lemma \ref{lema},
$$y_{k+1}=(y_0+\sum_{i=0}^{k}h_i)\prod_{j=0}^{k}f_j,$$
where $s_0=1$, and for $k\geq 0$, $s_{k+1}=(\prod\limits_{j=0}^{k}f_j)^{-1}=(\prod\limits_{j=0}^{k}4)^{-1}=\frac{1}{4^{k+1}}$, and
\begin{align*}
h_k=s_{k+1}g_k=\frac{1}{4^{k+1}}[4R^*(S^{k}(G))+(m_{k}+n_{k})(m_{k}-n_{k}+1)+2m_{k}(m_{k}-n_{k})].
\end{align*}
By Theorem \ref{sk*},
\begin{align*}
h_k&=\frac{1}{4^{k+1}}\big\{4[8^kR^*(G)+\frac{8^k-2^k}{3}m(2m-2n+1)]+(m_{k}+n_{k})(m_{k}-n_{k}+1)+2m_{k}(m_{k}-n_{k})\\
&=\frac{1}{4^{k+1}}\big[4\cdot8^kR^*(G)+\frac{4(8^k-2^k)}{3}m(2m-2n+1)+(2^km+n+(2^k-1)m)(m-n+1)\\
&~~~~+2\cdot 2^km(m-n)\big]\\
&=2^kR^*(G)+\frac{8^k-2^k}{3\cdot4^k}m(2m-2n+1)+\frac{(2^{k+1}m-m+n)(m-n+1)}{4^{k+1}}+\frac{2^{k+1}m(m-n)}{4^{k+1}}\\
&=2^kR^*(G)+\frac{8^k-2^k}{3\cdot4^k}m(2m-2n+1)+\frac{2^{k+1}m(2m-2n+1)-(m-n)(m-n+1)}{4^{k+1}}\\
&=2^kR^*(G)+(\frac{2^k-(\frac{1}{2})^k}{3}+\frac{1}{2^{k+1}})m(2m-2n+1)-\frac{(m-n)(m-n+1)}{4^{k+1}}.
\end{align*}
Hence
\begin{align*}
&y_{k+1}=(y_0+\sum_{i=0}^{k}h_i)\prod_{j=0}^{k}f_j\\
&=\big\{y_0+\sum_{i=0}^{k}\big[2^iR^*(G)+(\frac{2^i-(\frac{1}{2})^i}{3}+\frac{1}{2^{i+1}})m(2m-2n+1)-\frac{(m-n)(m-n+1)}{4^{i+1}}\big]\big\}\prod_{j=0}^{k}4\\
&=4^{k+1}\big[R^+(G)+(2^{k+1}-1)R^*(G)+\frac{2^{k+1}-(\frac{1}{2})^{k+1}}{3}m(2m-2n+1)\\
&~~~~-\frac{1-(\frac{1}{4})^{k+1}}{3}(m-n)(m-n+1)\big]\\
&=4^{k+1}R^+(G)+(8^{k+1}-4^{k+1})R^*(G)+\frac{8^{k+1}-2^{k+1}}{3}m(2m-2n+1)\\
&~~~~-\frac{4^{k+1}-1}{3}(m-n)(m-n+1).
\end{align*}
Thus the result follows.
\end{proof}

Finally, we compute $R(S^k(G))$ using the results of Theorems \ref{sk*} and \ref{sk+}, still along with Lemma \ref{lema}.

\begin{thm}
Let $G$ be a connected graph with $n\geq 2$ vertices and $m$ edges. Then
\begin{align}
R(S^k(G))=&2^kR(G)+\frac{4^{k}-2^{k}}{2}R^+(G)+\frac{8^{k}-2\cdot 4^{k}+2^{k}}{4}R^*(G)\notag\\
&+\frac{8^k-2^k}{12}m(2m-2n+1)-\frac{4^k-1}{6}(m-n)(m-n+1).
\end{align}
\end{thm}
\begin{proof}
By Theorem \ref{main}, for $k\geq 1$,
\begin{align*}
R({S^k(G)})=2R(S^{k-1}(G))+R^+(S^{k-1}(G))+\frac{1}{2}R^*(S^{k-1}(G))+\frac{m^2_{k-1}-n^2_{k-1}+n_{k-1}}{2}.
\end{align*}
Set $y_k=R(S^k(G))$, $f_k=2$, $g_k=R^+(S^{k}(G))+\frac{1}{2}R^*(S^{k}(G))+\frac{m^2_{k}-n^2_{k}+n_{k}}{2}$, $k=0,1,2,\ldots$. Then we have
$$y_{k+1}=f_ky_k+g_k,~~y_0=R^+(G),~~k\geq 0.$$
By Theorems \ref{sk*} and \ref{sk+}, and bearing in mind that $m_k=2^km$ and $n_k=n+(2^k-1)m$, we have
\begin{align*}
&g_k=R^+(S^{k}(G))+\frac{1}{2}R^*(S^{k}(G))+\frac{m^2_{k}-n^2_{k}+n_{k}}{2}\\
&=[4^kR^+(G)+(8^k-4^k)R^*(G)+\frac{8^k-2^k}{3}m(2m-2n+1)-\frac{4^k-1}{3}(m-n)(m-n+1)]\\
&~~+\frac{1}{2}[8^kR^*(G)+\frac{8^k-2^k}{3}m(2m-2n+1)]+\frac{(2^km)^2-[n+(2^k-1)m]^2+n+(2^k-1)m}{2}\\
&=4^kR^+(G)+(\frac{3}{2}8^k-4^k)R^*(G)+\frac{8^k}{2}m(2m-2n+1)-(\frac{4^k}{3}+\frac{1}{6})(m-n)(m-n+1).
\end{align*}
Then by Lemma \ref{lema},
$$y_{k+1}=(y_0+\sum_{i=0}^{k}h_i)\prod_{j=0}^{k}f_j,$$
where $s_0=1$ and for $k\geq 0$, $s_{k+1}=(\prod\limits_{j=0}^{k}f_j)^{-1}=(\prod\limits_{j=0}^{k}2)^{-1}=\frac{1}{2^{k+1}}$, and
\begin{align*}
&h_k=s_{k+1}g_k\\
&=\frac{1}{2^{k+1}}[4^kR^+(G)+(\frac{3}{2}8^k-4^k)R^*(G)+\frac{8^k}{2}m(2m-2n+1)-(\frac{4^k}{3}+\frac{1}{6})(m-n)(m-n+1)]\\
&=2^{k-1}R^+(G)+(3\cdot4^{k-1}-2^{k-1})R^*(G)+4^{k-1}m(2m-2n+1)\\
&~~~~-\frac{2^{k-1}+(\frac{1}{2})^{k+2}}{3}(m-n)(m-n+1).
\end{align*}
Hence
\begin{align*}
&y_{k+1}=(y_0+\sum_{i=0}^{k}h_i)\prod_{j=0}^{k}f_j\\
&=2^{k+1}\big\{R^+(G)+\sum_{i=0}^{k}[2^{i-1}R^+(G)+(3\cdot4^{i-1}-2^{i-1})R^*(G)+4^{i-1}m(2m-2n+1)\\
&~~~~-\frac{2^{i-1}+(\frac{1}{2})^{i+2}}{3}(m-n)(m-n+1)]\big\}\\
&=2^{k+1}\big\{R^+(G)+\frac{2^{k+1}-1}{2}R^+(G)+(4^k-\frac{1}{4}-2^k+\frac{1}{2})R^*(G)+\frac{4^k-\frac{1}{4}}{3}m(2m-2n+1)\\
&~~~~-\frac{2^{k}-\frac{1}{2}+\frac{1}{2}-(\frac{1}{2})^{k+2}}{3}(m-n)(m-n+1)]\big\}\\
&=2^{k+1}R^+(G)+\frac{4^{k+1}-2^{k+1}}{2}R^+(G)+\frac{8^{k+1}-2\cdot 4^{k+1}+2^{k+1}}{4}R^*(G)\\
&~~~~+\frac{8^{k+1}-2^{k+1}}{12}m(2m-2n+1)-\frac{4^{k+1}-1}{6}(m-n)(m-n+1).
\end{align*}
Then $R(S^k(G))$ is obtained by noticing that $R(S^k(G))=y_k$.
\end{proof}
\section{Formulae for $R(T(G))$, $R^+(T(G))$, and $R^*(T(G))$}
Resistance distances for $T(G)$ are obtained in \cite{hzb} as given in first lemma following. Since $T(G)$ can be obtained from $G$ by adding, for each edge $uv\in E$,  an additional vertex whose neighbors are $u$ and $v$, it follows that $T(G)$ has the same vertex set as $S(G)$ and can be written as $V(T(G))=V\cup V'$, where $V'$ denotes the set of newly added vertices. Consequently, it follows that $|V'|=|E|=m$ and $|V(T(G))|=n+m$. For clarity, we use $\Omega^T_{i,j}$ to denote resistance distances between vertices $i$ and $j$ in $T(G)$.
\begin{Lemma}\label{rgrs}
Resistance distances in $T(G)$ are given as:

(1) For $i,j\in V$,  $$\Omega_{ij}^T=\frac{2}{3}\Omega_{ij}.$$

(2) For $i\in V'$, $\Gamma(i)=\{k,l\}$, and $j\in V$,
$$\Omega_{ij}^T=\frac{1}{2}+\frac{1}{3}\Omega_{kj}+\frac{1}{3}\Omega_{lj}-\frac{1}{6}\Omega_{kl}.$$

(3) For $i,j\in V'$, $\Gamma(i)=\{k,l\}$, $\Gamma(j)=\{p,q\}$,
$$\Omega_{ij}^T=1+\frac{\Omega_{pk}+\Omega_{qk}+\Omega_{pl}+\Omega_{ql}-\Omega_{kl}-\Omega_{pq}}{6}.$$
\end{Lemma}
Comparing results in Lemma \ref{sg} with that in Lemma \ref{rgrs}, one immediately obtains:
\begin{pro}\label{pc}
Resistance distances in $S(G)$ and $T(G)$ relate as:

(1) For $i,j\in V$,  $$\Omega_{ij}^T=\frac{1}{3}\Omega_{ij}^S.$$

(2) For $i\in V'$, $j\in V$, $$\Omega_{ij}^T=\frac{1}{3}\Omega_{ij}^S+\frac{1}{3}.$$

(3) For $i,j\in V'$, $$\Omega_{ij}^T=\frac{1}{3}\Omega_{ij}^S+\frac{2}{3}.$$
\end{pro}

Proposition \ref{pc} not only compares resistance distances in $S(G)$ and $T(G)$, but also leads to a convenient way to compute $R(T(G))$, $R^+(T(G))$, and $R^*(T(G))$.

First we treat the Kirchhoff index of $T(G)$.
\begin{thm}\label{t}
Let $G$ be a connected graph with $n\geq 2$ vertices and $m$ edges. Then
\begin{equation}
R(T(G))=\frac{2}{3}R(G)+\frac{1}{3}R^+(G)+\frac{1}{6}R^*(G)+\frac{3m^2-n^2+2mn-2m+n}{6}.
\end{equation}
\end{thm}
\begin{proof}
Since $V(T(G))=V\cup V'$, by the definition of the Kirchhoff index,
\begin{align*}
R(T(G))=\sum_{\{i,j\}\subseteq (V(T(G))}\Omega^T_{ij}=\sum_{\{i,j\}\subseteq (V\cup V')}\Omega^T_{ij}
=\sum_{\{i,j\}\subseteq V}\Omega^T_{ij}+\sum_{i\in V'}\sum_{j\in V}\Omega^T_{ij}+\sum_{\{i,j\}\subseteq V'}\Omega^T_{ij}.
\end{align*}
Then by Proposition \ref{pc}, it follows that
\begin{align*}
R(T(G))&=\sum_{\{i,j\}\subseteq V}\frac{1}{3}\Omega^S_{ij}+\sum_{i\in V'}\sum_{j\in V}(\frac{1}{3}\Omega^S_{ij}+\frac{1}{3})+\sum_{\{i,j\}\subseteq V'}(\frac{1}{3}\Omega^S_{ij}+\frac{2}{3})\\
&=\frac{1}{3}[\sum_{\{i,j\}\subseteq V}\Omega^S_{ij}+\sum_{i\in V'}\sum_{j\in V}\Omega^S_{ij}+\sum_{\{i,j\}\subseteq V'}\Omega^S_{ij}]+\sum_{i\in V'}\sum_{j\in V}\frac{1}{3}+\sum_{\{i,j\}\subseteq V'}\frac{2}{3}\\
&=\frac{1}{3}R(S(G))+\frac{mn}{3}+\frac{m(m-1)}{3}.
\end{align*}
Then the desired result is obtained simply by substitution of the formula of $R(S(G))$ into the above equation.
\end{proof}

\textit{Remark.} For a regular graph $G$ with regularity $r$, it is shown in \cite{wyl} that
$$R(T(G))=\frac{(r+2)^2}{6}R(G)+\frac{(n^2-n)(r+2)}{6}+\frac{n^2(r^2-4)}{8}+\frac{n}{2}.$$
In fact, this follows directly from Theorem \ref{t} upon noticing that $R^+(G)=2rR(G)$, $R^*(G)=r^2R(G)$, and $m=\frac{nr}{2}$. For a general graph $G$, it is shown in \cite{hzb} that
\begin{align*}
R(T(G))&=\frac{2(n^2-m^2)}{3n^2}R(G)+\frac{(m+n)(3m-n+1)-3m-\pi^TL_G^\#\pi}{6}\\
&+\frac{m+n}{3n}\sum_{i=1}^{n}d_iR_i(G),
\end{align*}
where $\pi=(d_1,d_2,\ldots,d_n)^T$ is the column vector of the degree sequence of $G$, $L_G^\#$ is the group inverse of the Laplacian matrix of $G$, and $R_i(G)$ is the sum of resistance distances between $i$ and all other vertices of $G$.
Our result of Theorem 4.3 appears more ``neat".

Now we treat $R^+(T(G))$, using $d_i^T$ to denote the degree of $i$ in $T(G)$. It is seen by the definition of $T(G)$ that for $i\in V$, $d_i^T=2d_i$, and for $i\in V'$, $d^T_i=2$.
\begin{thm}\label{t+}
Let $G$ be a connected graph with $n\geq 2$ vertices and $m$ edges. Then
\begin{equation}
R^+(T(G))=2R^+(G)+2R^*(G)+\frac{12m^2-mn-n^2-2m+n}{3}.
\end{equation}
\end{thm}
\begin{proof}
By the definition of the additive degree-Kirchhoff index and Proposition \ref{pc}, we have
\begin{align*}
R^+(T(G))=\sum_{\{i,j\}\subseteq V}(d^T_i+d^T_j)\Omega^T_{ij}+\sum_{i\in V'}\sum_{j\in V}(d^T_i+d^T_j)\Omega^T_{ij}+\sum_{\{i,j\}\subseteq V'}(d^T_i+d^T_j)\Omega^T_{ij}.
\end{align*}
Then according to the structure of $T(G)$ and Proposition \ref{pc},
\begin{align*}
R^+(T(G))&=\sum_{\{i,j\}\subseteq V}(2d_i+2d_j)\frac{1}{3}\Omega^S_{ij}+\sum_{i\in V'}\sum_{j\in V}(2+2d_j)[\frac{1}{3}\Omega^S_{ij}+\frac{1}{3}]+\sum_{\{i,j\}\subseteq V'}(2+2)[\frac{1}{3}\Omega^S_{ij}+\frac{2}{3}]\\
&=\frac{2}{3}\sum_{\{i,j\}\subseteq V}(d_i+d_j)\Omega^S_{ij}+\frac{2}{3}\sum_{i\in V'}\sum_{j\in V}(1+d_j)[\Omega^S_{ij}+1]+\frac{4}{3}\sum_{\{i,j\}\subseteq V'}[\Omega^S_{ij}+2]\\
&=\frac{2}{3}\sum_{\{i,j\}\subseteq V}(d_i+d_j)\Omega^S_{ij}+\frac{2}{3}\sum_{i\in V'}\sum_{j\in V}\Omega^S_{ij}+\frac{2}{3}\sum_{i\in V'}\sum_{j\in V}d_j\Omega^S_{ij}+\frac{4}{3}\sum_{\{i,j\}\subseteq V'}\Omega^S_{ij}\\
&~~~~+\frac{2}{3}\sum_{i\in V'}\sum_{j\in V}d_j+\frac{2}{3}\sum_{i\in V'}\sum_{j\in V}1+\frac{4}{3}\sum_{\{i,j\}\subseteq V'}2\\
&=\frac{2}{3}\sum_{\{i,j\}\subseteq V}(d_i+d_j)\Omega^S_{ij}+\frac{2}{3}\sum_{i\in V'}\sum_{j\in V}\Omega^S_{ij}+\frac{2}{3}\sum_{i\in V'}\sum_{j\in V}d_j\Omega^S_{ij}+\frac{4}{3}\sum_{\{i,j\}\subseteq V'}\Omega^S_{ij}\\
&~~~~+\frac{4m^2}{3}+\frac{2mn}{3}+\frac{4m(m-1)}{3}.
\end{align*}
Noticing that the first four terms in the above equality are given in Eqs. (\ref{e12}), (\ref{e14}), (\ref{e17}), and (\ref{e18}), the the desired result is obtained by substitution of these results into this last form for $R^+(T(G))$.
\end{proof}
\begin{thm}\label{t*}
Let $G$ be a connected graph with $n\geq 2$ vertices and $m$ edges. Then
\begin{equation}
R^*(T(G))=8R^*(G)+8m^2-4mn.
\end{equation}
\end{thm}
\begin{proof}
By Proposition \ref{pc},
\begin{align*}
R^*(T(G))&=\sum_{\{i,j\}\subseteq V\cup V'}d^T_id^T_j\Omega^T_{ij}=\sum_{\{i,j\}\subseteq V}d^T_id^T_j\Omega^T_{ij}+\sum_{i\in V'}\sum_{j\in V}d^T_id^T_j\Omega^T_{ij}+\sum_{\{i,j\}\subseteq V'}d^T_id^T_j\Omega^T_{ij}\notag\\
&=\sum_{\{i,j\}\subseteq V}2d_i2d_j\frac{1}{3}\Omega^S_{ij}+\sum_{i\in V'}\sum_{j\in V}2\cdot2d_j[\frac{1}{3}\Omega^S_{ij}+\frac{1}{3}]+\sum_{\{i,j\}\subseteq V'}2\cdot2[\frac{1}{3}\Omega^S_{ij}+\frac{2}{3}]\\
&=\frac{4}{3}\sum_{\{i,j\}\subseteq V}d_id_j\Omega^S_{ij}+\frac{4}{3}\sum_{i\in V'}\sum_{j\in V}d_j\Omega^S_{ij}+\frac{4}{3}\sum_{\{i,j\}\subseteq V'}\Omega^S_{ij}+\frac{4}{3}\sum_{i\in V'}\sum_{j\in V}d_j+\frac{8}{3}\sum_{\{i,j\}\subseteq V'}1\\
&=\frac{4}{3}\sum_{\{i,j\}\subseteq V}d_id_j\Omega^S_{ij}+\frac{4}{3}\sum_{i\in V'}\sum_{j\in V}d_j\Omega^S_{ij}+\frac{4}{3}\sum_{\{i,j\}\subseteq V'}\Omega^S_{ij}+\frac{8m^2}{3}+\frac{4m(m-1)}{3}.
\end{align*}
Noticing that $\sum\limits_{\{i,j\}\subseteq V}d_id_j\Omega^S_{ij}=2\sum\limits_{\{i,j\}\subseteq V}d_id_j\Omega_{ij}=2R^*(G)$, by Eqs. (\ref{e14}) and (\ref{e17}), we obtain the required result.
\end{proof}

From these proceeding results, comparison of our triplet of Kirchhoffian invariants of $S(G)$ and $T(G)$ result:
\begin{pro}
Let $G$ be a connected graph with $n$ vertices and $m$ edges. Then
\begin{align}
R(T(G))&=\frac{1}{3}R(S(G))+\frac{m(m+n-1)}{3}.\\
R^+(T(G))&=\frac{1}{2}R^+(S(G))+\frac{15m^2+n^2+4mn-7m-n}{6}.\\
R^*(T(G))&=R^*(S(G))+4m^2-2m.
\end{align}
\end{pro}

\section{Formulae for our three Kirchhoffian invariants of iterated triangulations}
Set $T^0(G)=G$, $T^k(G)=T(T^{k-1}(G))$, $k=1,2,\ldots$. For $k\geq 0$, in this section, in the proofs, denote the number of vertices and edges of $T^k(G)$ by $n_k$ and $m_k$, respectively. Then it is easily verified that $n_0=n$, $m_0=m$, and for $k\geq 1$,
\begin{equation}\label{mknk}
n_k=\frac{3^k-1}{2}m+n,\quad m_k=3^km.
\end{equation}

Now we are ready to treat the three graph invariants of iterated triangulations, first for $R^*(T^k(G))$.
\begin{thm}\label{tk*}
Let $G$ be a connected graph with $n\geq 2$ vertices and $m$ edges. Then
\begin{equation}\label{e}
R^*(T^k(G))=8^kR^*(G)+(6\cdot 9^k-\frac{2}{5}3^k-\frac{28}{5}8^k)m^2-\frac{4}{5}(8^k-3^k)mn.
\end{equation}
\end{thm}
\begin{proof}
Let $n_k$ and $m_k$ be defined as in Eq. (\ref{mknk}). Then $T^k(G)$ is the graph with $n_k$ vertices and $m_k$ edges. By Theorem \ref{t*},
$$R^*(T^k(G))=8R^*(T^{k-1}(G))+8m_{k-1}^2-4m_{k-1}n_{k-1},~~k\geq 1.$$
Set $y_k=R^*(T^k(G))$, $f_k=8$, $g_k=8m_{k}^2-4m_{k}n_{k}$, $k=0,1,2,\ldots$. Then we have
$$y_{k+1}=f_ky_k+g_k,~~y_0=R^*(G),~~k\geq 0.$$
It is easy to compute that
\begin{align*}
g_k&=8m_{k}^2-4m_{k}n_{k}=4m_k(2m_k-n_k)=4\cdot 3^km(2\cdot 3^km-\frac{3^k-1}{2}m-n)\notag\\
&=(6\cdot 9^k+2\cdot 3^k)m^2-4\cdot 3^kmn.
\end{align*}
By Lemma \ref{lema},
$$y_{k+1}=(y_0+\sum_{i=0}^{k}h_i)\prod_{j=0}^{k}f_j,$$
where $s_0=1$, and for $k\geq 0$, $s_{k+1}=(\prod\limits_{j=0}^{k}f_j)^{-1}=(\prod\limits_{j=0}^{k}8)^{-1}=\frac{1}{8^{k+1}}$, and
$$h_k=s_{k+1}g_k=\frac{(6\cdot 9^k+2\cdot 3^k)m^2-4\cdot 3^kmn}{8^{k+1}}=[\frac{3}{4}(\frac{9}{8})^k+\frac{1}{4}(\frac{3}{8})^k]m^2-\frac{1}{2}(\frac{3}{8})^kmn.$$
Hence
\begin{align*}
y_{k+1}&=\big[y_0+\sum_{i=0}^{k}\big([\frac{3}{4}(\frac{9}{8})^i+\frac{1}{4}(\frac{3}{8})^i]m^2-\frac{1}{2}(\frac{3}{8})^imn\big)\big]\prod_{j=0}^{k}8\\
&=8^{k+1}[R^*(G)+\frac{3}{4}\frac{(\frac{9}{8})^{k+1}-1}{\frac{1}{8}}m^2+\frac{1}{4}\frac{1-(\frac{3}{8})^{k+1}}{\frac{5}{8}}m^2-\frac{1}{2}\frac{1-(\frac{3}{8})^{k+1}}{\frac{5}{8}}mn]\\
&=8^{k+1}R^*(G)+(6\cdot 9^{k+1}-\frac{2}{5}3^{k+1}-\frac{28}{5}8^{k+1})m^2-\frac{4}{5}(8^{k+1}-3^{k+1})mn.
\end{align*}
Thus the proof is completed.
\end{proof}

Now we treat $R(T^k(G))$.

\begin{thm}\label{tk+}
Let $G$ be a connected graph with $n\geq 2$ vertices and $m$ edges. Then
\begin{align}
R^+(T^k(G))=&2^kR^+(G)+\frac{8^k-2^k}{3}R^*(G)+(\frac{9}{4}9^k-\frac{28}{15}8^k-\frac{7}{15}3^k)m^2-(\frac{4}{15}8^k-\frac{14}{15}3^k+\frac{2^k}{3})mn\notag\\
&-(\frac{3^k}{2}-\frac{2^k}{3})m-\frac{2^k(n^2-n)}{3}+\frac{(m-2n)(m-2n+2)}{12}.
\end{align}
\end{thm}
\begin{proof}
By Theorem \ref{t+}, for $k\geq 1$,
\begin{align*}
R^+({T^k(G)})=2R^+(T^{k-1}(G))+2R^*(T^{k-1}(G))+\frac{12m_{k-1}^2-m_{k-1}n_{k-1}-n_{k-1}^2-2m_{k-1}+n_{k-1}}{3}.
\end{align*}
Set $y_k=R^+(T^k(G))$, $f_k=2$, $g_k=2R^*(T^{k}(G))+\frac{12m_{k}^2-m_{k}n_{k}-n_{k}^2-2m_{k}+n_{k}}{3}$, $k=0,1,2,\ldots$. Then we have
$$y_{k+1}=f_ky_k+g_k,~~y_0=R^+(G),~~k\geq 0.$$
By Theorem \ref{tk*}, it is not difficult to compute that
\begin{align*}
g_k=&2R^*(T^{k}(G))+\frac{12m_{k}^2-m_{k}n_{k}-n_{k}^2-2m_{k}+n_{k}}{3}\\
=&2[8^kR^*(G)+(6\cdot 9^k-\frac{2}{5}3^k-\frac{28}{5}8^k)m^2-\frac{4}{5}(8^k-3^k)mn]\\
&+\frac{12\cdot(3^km)^2-3^km(\frac{3^km-1}{2}+n)-(\frac{3^km-1}{2}+n)^2-2\cdot 3^km+(\frac{3^km-1}{2}+n)}{3}\\
=&2\cdot 8^kR^*(G)+(\frac{63}{4}9^k-\frac{56}{5}8^k-\frac{7}{15}3^k)m^2-(\frac{8}{5}8^k-\frac{14}{15}3^k)mn-\frac{3^k}{2}m\\
&-\frac{(m-2n)(m-2n+2)}{12}.
\end{align*}
By Lemma \ref{lema},
$$y_{k+1}=(y_0+\sum_{i=0}^{k}h_i)\prod_{j=0}^{k}f_j,$$
where $s_0=1$, and for $k\geq 0$, $s_{k+1}=(\prod_{j=0}^{k}f_j)^{-1}=(\prod_{j=0}^{k}2)^{-1}=\frac{1}{2^{k+1}}$, and
\begin{align*}
h_k=&s_{k+1}g_k=\frac{1}{2^{k+1}}[2\cdot 8^kR^*(G)+(\frac{63}{4}9^k-\frac{56}{5}8^k-\frac{7}{15}3^k)m^2-(\frac{8}{5}8^k-\frac{14}{15}3^k)mn-\frac{3^k}{2}m\\
&-\frac{(m-2n)(m-2n+2)}{12}]\\
=&4^kR^*(G)+[\frac{63}{8}(\frac{9}{2})^k-\frac{28}{5}4^k-\frac{7}{30}(\frac{3}{2})^k]m^2-[\frac{4}{5}4^k-\frac{7}{15}(\frac{3}{2})^k]mn-\frac{1}{4}(\frac{3}{2})^km\\
&-\frac{1}{2^{k+1}}\frac{(m-2n)(m-2n+2)}{12}.
\end{align*}
Hence
\begin{align*}
&y_{k+1}=(y_0+\sum_{i=0}^{k}h_i)\prod_{j=0}^{k}f_j\\
=&2^{k+1}\big\{R^+(G)+\sum_{i=0}^{k}\big[4^iR^*(G)+\big(\frac{63}{8}(\frac{9}{2})^i-\frac{28}{5}4^i-\frac{7}{30}(\frac{3}{2})^i\big)m^2-\big(\frac{4}{5}4^i-\frac{7}{15}(\frac{3}{2})^i\big)mn\\
&-\frac{1}{4}(\frac{3}{2})^im-\frac{(m-2n)(m-2n+2)}{2^{i+1}}\big]\big\}\\
=&2^{k+1}R^+(G)+2^{k+1}\big\{\frac{{4^{k+1}-1}}{3}R^*(G)+\big[\frac{9}{4}\big((\frac{9}{2})^{k+1}-1\big)-\frac{28}{15}(4^{k+1}-1)-\frac{7}{15}\big((\frac{3}{2})^{k+1}-1\big)\big]m^2\\
&-\big[\frac{4}{15}(4^{k+1}-1)-\frac{14}{15}\big((\frac{3}{2})^{k+1}-1\big)\big]mn-\frac{1}{2}[(\frac{3}{2})^{k+1}-1]m\\
&-[1-(\frac{1}{2})^{k+1}]\frac{(m-2n)(m-2n+2)}{12}\big\}\\
=&2^{k+1}R^+(G)+\frac{8^{k+1}-2^{k+1}}{3}R^*(G)+(\frac{9}{4}9^{k+1}-\frac{28}{15}8^{k+1}-\frac{7}{15}3^{k+1})m^2-(\frac{3^{k+1}}{2}-\frac{2^{k+1}}{3})m\\
&-(\frac{4}{15}8^{k+1}-\frac{14}{15}3^{k+1}+\frac{2^{k+1}}{3})mn-\frac{2^{k+1}}{3}n^2+\frac{2^{k+1}}{3}n+\frac{(m-2n)(m-2n+2)}{12}.
\end{align*}
Then set $R^+(T^k(G))=y_k$ to yield the desired result.
\end{proof}

We end this section by treating $R(T^k(G))$.

\begin{thm}\label{tk}
Let $G$ be a connected graph with $n\geq 2$ vertices and $m$ edges. Then
\begin{align}
&R(T^k(G))=(\frac{2}{3})^kR(G)+[\frac{2^k}{4}-\frac{1}{4}(\frac{2}{3})^k]R^+(G)+[\frac{5}{132}8^k-\frac{2^k}{12}+\frac{1}{22}(\frac{2}{3})^{k}]R^*(G)\notag\\
&+[\frac{57}{200}9^k-\frac{7}{3}8^k-\frac{11}{84}3^k+\frac{2683}{46200}(\frac{2}{3})^k]m^2-[\frac{8^k}{33}-\frac{11}{42}3^k+\frac{2^k}{12}+\frac{137}{924}(\frac{2}{3})^k]mn\notag\\
&-[2^k-(\frac{2}{3})^k]\frac{n^2-n}{12}-[\frac{5}{28}3^k-\frac{2^k}{12}-\frac{2}{21}(\frac{2}{3})^k]m-[1-(\frac{2}{3})^k]\frac{(m-2n)(m-2n+2)}{24}.
\end{align}
\end{thm}
\begin{proof}
By Theorem \ref{t},
\begin{align*}
R({T^k(G)})=&\frac{2}{3}R(T^{k-1}(G))+\frac{1}{3}R^+(T^{k-1}(G))+\frac{1}{6}R^*(T^{k-1}(G))\\
&+\frac{3m_{k-1}^2-n_{k-1}^2+2m_{k-1}n_{k-1}-2m_{k-1}+n_{k-1}}{6},\quad k\geq 1.
\end{align*}
Set $y_k=R(T^k(G))$, $f_k=\frac{2}{3}$, $g_k=\frac{1}{3}R^+(T^{k}(G))+\frac{1}{6}R^*(T^{k}(G))+\frac{3m_{k}^2-n_{k}^2+2m_{k}n_{k}-2m_{k}+n_{k}}{6}$, $k=0,1,2,\ldots$. Then we have
$$y_{k+1}=f_ky_k+g_k,~~y_0=R(G),~~k\geq 0.$$
By Theorems \ref{tk*} and \ref{tk+}, it is not difficult to compute that
\begin{align*}
g_k=&\frac{1}{3}R^+(T^{k}(G))+\frac{1}{6}R^*(T^{k}(G))+\frac{3m_{k}^2-n_{k}^2+2m_{k}n_{k}-2m_{k}+n_{k}}{6}\\
=&\frac{1}{3}[2^kR^+(G)+\frac{8^k-2^k}{3}R^*(G)+(\frac{9}{4}9^k-\frac{28}{15}8^k-\frac{7}{15}3^k)m^2-(\frac{4}{15}8^k-\frac{14}{15}3^k+\frac{2^k}{3})mn\notag\\
&-(\frac{3^k}{2}-\frac{2^k}{3})m-\frac{2^k(n^2-n)}{3}+\frac{(m-2n)(m-2n+2)}{12}]\\
&+\frac{1}{6}[8^kR^*(G)+(6\cdot 9^k-\frac{2}{5}3^k-\frac{28}{5}8^k)m^2-\frac{4}{5}(8^k-3^k)mn]\\
&+\frac{3(3^km)^2-(\frac{3^k-1}{2}m+n)^2+2\cdot 3^km(\frac{3^k-1}{2}m+n)-2\cdot3^km+\frac{3^k-1}{2}m+n}{6}\\
=&\frac{2^k}{3}R^+(G)+(\frac{5}{18}8^k-\frac{2^k}{9})R^*(G)+(\frac{19}{8}9^k-\frac{14}{9}8^k-\frac{11}{36}3^k)m^2-(\frac{2}{9}8^k-\frac{11}{18}3^k+\frac{2^k}{9})mn\\
&-(\frac{5}{12}3^k-\frac{2^k}{9})m-2^k\frac{n^2-n}{9}-\frac{(m-2n)(m-2n+2)}{72}.
\end{align*}
By Lemma \ref{lema},
$$y_{k+1}=(y_0+\sum_{i=0}^{k}h_i)\prod_{j=0}^{k}f_j,$$
where $s_0=1$ and for $k\geq 0$, $s_{k+1}=(\prod\limits_{j=0}^{k}f_j)^{-1}=(\prod\limits_{j=0}^{k}\frac{2}{3})^{-1}=(\frac{3}{2})^{k+1}$, and
\begin{align*}
h_k=&s_{k+1}g_k=(\frac{3}{2})^{k+1}[\frac{2^k}{3}R^+(G)+(\frac{5}{18}8^k-\frac{2^k}{9})R^*(G)+(\frac{19}{8}9^k-\frac{14}{9}8^k-\frac{11}{36}3^k)m^2\\
&-(\frac{2}{9}8^k-\frac{11}{18}3^k+\frac{2^k}{9})mn-(\frac{5}{12}3^k-\frac{2^k}{9})m-2^k\frac{n^2-n}{9}-\frac{(m-2n)(m-2n+2)}{72}]\\
=&\frac{3^k}{2}R^+(G)+(\frac{5}{12}12^k-\frac{3^k}{6})R^*(G)+[\frac{57}{16}(\frac{27}{2})^k-\frac{7}{3}12^k-\frac{11}{24}(\frac{9}{2})^k]m^2\\
&-[\frac{12^k}{3}-\frac{11}{12}(\frac{9}{2})^k+\frac{3^k}{6}]mn-[\frac{5}{8}(\frac{9}{2})^k-\frac{3^k}{6}]m-\frac{3^k(n^2-n)}{6}-(\frac{3}{2})^k\frac{(m-2n)(m-2n+2)}{48}.
\end{align*}
Hence
\begin{align*}
&y_{k+1}=(y_0+\sum_{i=0}^{k}h_i)\prod_{j=0}^{k}f_j\\
=&(\frac{2}{3})^{k+1}\big\{R(G)+\sum_{i=0}^{k}\big(\frac{3^i}{2}R^+(G)+(\frac{5}{12}12^i-\frac{3^i}{6})R^*(G)+[\frac{57}{16}(\frac{27}{2})^i-\frac{7}{3}12^i-\frac{11}{24}(\frac{9}{2})^i]m^2\\
&-[\frac{12^i}{3}-\frac{11}{12}(\frac{9}{2})^i+\frac{3^i}{6}]mn-[\frac{5}{8}(\frac{9}{2})^i-\frac{3^i}{6}]m-\frac{3^i(n^2-n)}{6}-(\frac{3}{2})^i\frac{(m-2n)(m-2n+2)}{48}\big)\big\}\\
=&(\frac{2}{3})^{k+1}R(G)+(\frac{2}{3})^{k+1}\big\{\frac{3^{k+1}-1}{4}R^+(G)+[\frac{5}{132}(12^{k+1}-1)-\frac{3^{k+1}-1}{12}]R^*(G)\\
&+\big[\frac{57}{200}\big((\frac{27}{2})^{k+1}-1\big)-\frac{7}{33}(12^{k+1}-1)-\frac{11}{84}\big((\frac{9}{2})^{k+1}-1\big)\big]m^2-\big[\frac{1}{33}(12^{k+1}-1)\\
&-\frac{11}{42}\big((\frac{9}{2})^{k+1}-1\big)-\frac{3^{k+1}-1}{12}\big]mn-\big[\frac{5}{28}\big((\frac{9}{2})^{k+1}-1\big)-\frac{3^{k+1}-1}{12}\big]m\\
&-\frac{(3^{k+1}-1)(n^2-n)}{12}-\big((\frac{3}{2})^{k+1}-1\big)\frac{(m-2n)(m-2n+2)}{24}\big\}\\
=&(\frac{2}{3})^{k+1}R(G)+\big[\frac{2^{k+1}}{4}-\frac{1}{4}(\frac{2}{3})^{k+1}\big]R^+(G)+\big[\frac{5}{132}(8^{k+1}-(\frac{2}{3})^{k+1})-\frac{1}{12}(2^{k+1}-(\frac{2}{3})^{k+1})\big]R^*(G)\\
&+\big[\frac{57}{200}\big(9^{k+1}-(\frac{2}{3})^{k+1}\big)-\frac{7}{33}(8^{k+1}-(\frac{2}{3})^{k+1})-\frac{11}{84}\big(3^{k+1}-(\frac{2}{3})^{k+1}\big)\big]m^2-\big[\frac{1}{33}\big(8^{k+1}-(\frac{2}{3})^{k+1}\big)\\
&-\frac{11}{42}\big((3^{k+1}-(\frac{2}{3})^{k+1}\big)+\frac{1}{12}(2^{k+1}-(\frac{2}{3})^{k+1})\big]mn-\big[\frac{5}{28}(3^{k+1}-1)-\frac{1}{12}\big(2^{k+1}-(\frac{2}{3})^{k+1}\big)\big]m\\
&-\frac{(2^{k+1}-(\frac{2}{3})^{k+1})(n^2-n)}{12}-[1-(\frac{2}{3})^{k+1}]\frac{(m-2n)(m-2n+2)}{24}\\
=&(\frac{2}{3})^{k+1}R(G)+[\frac{2^{k+1}}{4}-\frac{1}{4}(\frac{2}{3})^{k+1}]R^+(G)+[\frac{5}{132}8^{k+1}-\frac{2^{k+1}}{12}+\frac{1}{22}(\frac{2}{3})^{k+1}]R^*(G)\notag\\
&+[\frac{57}{200}9^{k+1}-\frac{7}{3}8^{k+1}-\frac{11}{84}3^{k+1}+\frac{2683}{46200}(\frac{2}{3})^{k+1}]m^2-[2^{k+1}-(\frac{2}{3})^{k+1}]\frac{n^2-n}{12}\notag\\
&-[\frac{8^{k+1}}{33}-\frac{11}{42}3^{k+1}+\frac{2^{k+1}}{12}+\frac{137}{924}(\frac{2}{3})^{k+1}]mn-[\frac{5}{28}3^{k+1}-\frac{2^{k+1}}{12}-\frac{2}{21}(\frac{2}{3})^{k+1}]m\\
&-[1-(\frac{2}{3})^{k+1}]\frac{(m-2n)(m-2n+2)}{24}.
\end{align*}
Thus the desired result follows.
\end{proof}

From the formulas of Sections 4 and 5, it seems triangulation is more ``complicated" than subdivision--at least with regard to our Kirchhoffian graph invariants.

\section{Acknowledgement}
The authors acknowledge the support of the Welch Foundation of Houston, Texas through grant BD-0894. Y. Yang acknowledges the support of the National Science Foundation of China through Grant Nos. 11201404 and 11371307, China Postdoctoral Science Foundation through Grant Nos. 2012M521318 and 2013T60662, Special Funds for Postdoctoral Innovative Projects of Shandong Province through Grant No. 201203056.
%%References----------------------------------------------------


\begin{thebibliography}{6}
\bibitem{ah} M. Aouchiche, P. Hansen, On a conjecture about the Szeged index, European J. Combin. 31 (2010) 1662--1666.

\bibitem{bcpt} M. Bianchi, A. Cornaro, J.L. Palacios, A. Torriero, Bounds for the Kirchhoff index via majorization techniques, J. Math. Chem. 51 (2013) 569--587.

\bibitem{bcpt1} M. Bianchi, A. Cornaro, J.L. Palacios, A. Torriero, New Upper and Lower Bounds for the Additive Degree-Kirchhoff Index, Croat. Chem. Acta 86(4) (2013) 363--370.

\bibitem{ch2} N. Chair, Trigonometrical sums connected with the chiral Potts model, Verlinde dimension formula, two-dimensional resistor network, and number theory, Ann. Phys. 341 (2014) 56--76.


\bibitem{cz1} H. Chen, F. Zhang, Resistance distance and the normalized Laplacian
spectrum, Discrete Appl. Math. 155 (2007) 654--661.

\bibitem{cz} H. Chen, F. Zhang, Random walks and the effective resistance sum rules, Discrete Appl. Math. 158 (2010) 1691--1700.

\bibitem{cll} L. Chen, X. Li, M. Liu, The (revised) Szeged index and the Wiener index of a nonbipartite graph, European J. Combin. 36
(2014) 237--246.

\bibitem{das} K.C. Das, On the Kirchhoff index of graphs, Z. Naturforsch. 68a (2013) 531--538.

\bibitem{dc} Q. Deng, H. Chen, On the Kirchhoff index of the complement of a bipartite graph, Linear Algebra Appl.  439 (2013) 167--173.

\bibitem{dc1} Q. Deng, H. Chen, On extremal bipartite unicyclic graphs, Linear Algebra Appl.  444 (2014) 89--99.

\bibitem{ds} P.G. Doyle, J.L. Snell, Random walks and electric networks,
The Mathematical Association of America, Washington, DC 1984.

\bibitem{esv} W. Ellens, F.M. Spieksma, P. Van Mieghem, A. Jamakovic, R.E. Kooij, Effective
graph resistance, Linear Algebra Appl. 435 (2011) 2491--2506.

\bibitem{fgy} L. Feng, I. Gutman, G. Yu, Degree Kirchhoff Index of Unicyclic Graphs, MATCH Commun. Math. Comput. Chem. 69 (2013) 629--648.

\bibitem{f} R.M. Foster, The average impedance of an electrical network,
in: J.W. Edwards (ed.),  Contributions to Applied Mechanics, Ann
Arbor, Michigan, 1949, pp. 333--340.

\bibitem{gl} X. Gao, Y. Luo, W. Liu, Kirchhoff index in line, subdivision and total graphs of a regular graph, Discrete Appl. Math. 160 (2012) 560--565

\bibitem{gbs} A. Ghosh, S. Boyd, A. Saberi, Minimizing effective resistance of a
graph, SIAM Rev. 50(1) (2008) 37--66.

\bibitem{gf} I. Gutman, L. Feng, G. Yu, Degree resistance distance of unicyclic graphs, Trans. Comb. 1 (2012)
27--40.

\bibitem{gg} A.D. Gvishiani, V.A. Gurvich, Metric and ultrametric
spaces of resistances, in: Communications of the Moscow Mathematical Society, Russian Math.
Surveys 42 (6(258)) (1987) 235--236.

\bibitem{nag} M. Hakimi-Nezhaad, A.R. Ashrafi, I. Gutman, Note on degree Kirchhoff index of graphs, Trans. Comb. 2 (2013) 43--52.

\bibitem{hzb} S. Huang, J. Zhou, C. Bu, Some formulas for resistance distance and Kirchhoff index, manuscipt.

\bibitem{kya} M.H. Khalifeh, H. Yousefi-Azari, A.R. Ashrafi, S.G. Wagner, Some new results on distance-based graph invariants, European J. Combin. 30 (2009) 1149--1163.

\bibitem{ki} J. Kigami, Harmonic analysis for resistance forms, J. Funct. Anal. 204 (2003) 399--444.

\bibitem{kn} S. Klav\v{z}ar, M.J. Nadjafi-Arani, Wiener index in weighted graphs via unification of $\Theta^*$-classes, European J. Combin. 36 (2014) 71--76.

\bibitem{kn1} S. Klav\v{z}ar, M.J. Nadjafi-Arani, Improved bounds on the difference between the
Szeged index and the Wiener index of graphs, European J. Combin. 39 (2014) 148--156.

\bibitem{k} D.J. Klein, Graph geometry, graph metrics and Wiener, MATCH Commun. Math. Comput. Chem. 35 (1997) 7--27.

\bibitem{kr} D.J. Klein, M. Randi$\acute{\mbox{c}}$, Resistance distance, J. Math. Chem. 12 (1993) 81--95.

%\bibitem{ln} I. Lukovits, S. Nikoli\'{c} and N. Trinajsti\'{c}, Resistance distance in regular graphs, Int. J.
%Quantum Chem. 71 (1999) 217--225.

\bibitem{nash} C.St.J.A. Nash-Williams, Random walks and electric currents in networks, Proc. Cambridge Phil. Soc. 55 (1959) 181--194.

\bibitem{ns} A. Nikseresht, Z. Sepasdar, On the Kirchhoff and the Wiener Indices of Graphs and Block Decomposition, Electron J. Comb. 21(1) (2014) \#P1.25.

\bibitem{pa} J.L. Palacios, Upper and lower bounds for the additive degree-Kirchhoff index, MATCH Commun. Math. Comput. Chem. 70 (2013) 651--655.

\bibitem{pr} J.L. Palacios, J.M, Renom, Another look at the degree-Kirchhoff index, Int. J. Quantum Chem. 111(14) (2011) 3453--3455.

\bibitem{rk} V.R. Rosenfeld, D.J. Klein, An infinite family of graphs with a facile count of perfect matchings, Discrete Appl. Math. 166 (2014) 210--214.

\bibitem{sh} L.W. Shapiro, An electrical lemma, Math. Mag. 60 (1987) 36--38.

\bibitem{s1}G.E. Sharpe, Solution of the (m+1)-terminal resistive network problem by means of metric geometry, in: Proceedings of the First Asilomar Conference on Circuits and Systems, Pacific Grove, CA, November 1967, pp. 319--328.

\bibitem{s2} G.E. Sharpe, Theorem on resistive networks, Electron. Lett. 3 (1967) 444--445.

\bibitem{s3} G.E. Sharpe, Violation of the 2-triple property by resistive networks, Electron. Lett. 3 (1967) 543--544.

\bibitem{ss} G.E. Sharpe, B. Spain, On the solution of networks by means of the equicofactor matrix, IRE Trans. Circuit Theory 7 (1960) 230--239.

\bibitem{ss1} G.E. Sharpe, G.P.H. Styan, A note on equicofactor matrices, Proc. IEEE 55 (1967) 1226--1227.

\bibitem{ssn} M.H. Shirdareh-Haghighi, Z. Sepasdar, A. Nikseresht, On the Kirchoff index of graphs and some graph operations, Proc. Indian Acad. Sci., in press.

\bibitem{swzb} L. Sun, W. Wang, J. Zhou, C. Bu, Some results on resistance distances and resistance matrices, Linear Multilinear A., in press.

\bibitem{wyl} W. Wang, D. Yang, Y. Luo, The Laplacian polynomial and Kirchhoff index of graphs derived from regular graphs, Discrete Appl. Math. 161 (2013) 3063--3071.

\bibitem{xg} W. Xiao, I. Gutman, Resistance distance and Laplacian spectrum, Theor. Chem. Acc. 110 (2003) 284--289.

\bibitem{yy} W. Yan, Y. Yeh, On the number of matchings of a graph operator, Sci. China Ser. A Math. 49 (2006) 1381--1391.

\bibitem{yyz} W. Yan, Y. Yeh, F. Zhang, The asymptotic behavior of some indices of iterated line graphs of regular graphs, Discrete Appl. Math. 160 (2012) 1232--1239.

\bibitem{yang} Y. Yang, The Kirchhoff index of subdivisions of graphs, Discrete Appl. Math. (2014), http://dx.doi.org/10.1016/j.dam.2014.02.015.

\bibitem{yk} Y. Yang, D.J. Klein, A recursion formula for resistance distances and its appications, Discrete Appl. Math. 161 (2013) 2702--2715.

\bibitem{yyh} Z. You, L. You, W. Hong, Comment on ``Kirchhoff index in line, subdivision and total
graphs of a regular graph", Discrete Appl. Math. 161 (2013) 3100--3103.

\bibitem{zzh} Z. Zhang, Some physical and chemical indices of clique-inserted lattices, J. Stat. Mech. 2013 (2013) P10004.

\end{thebibliography}
\end{document}